\documentclass[aip,graphicx]{revtex4-1}

\usepackage{algorithm}
\usepackage{algpseudocode}
\usepackage{amssymb}
\usepackage{amsthm}
\usepackage{graphicx}

\draft 

\begin{document}

\title{The Beta-Wishart Ensemble} 

\author{Alexander Dubbs}
\email[]{dubbs@math.mit.edu}
\affiliation{Massachusetts Institute of Technology}

\author{Alan Edelman}
\email[]{edelman@math.mit.edu}
\affiliation{Massachusetts Institute of Technology}

\author{Plamen Koev}
\email[]{plamen.koev@sjsu.edu}
\affiliation{San Jose State University}

\author{Praveen Venkataramana}
\email[]{venkap@mit.edu}
\affiliation{Massachusetts Institute of Technology}

\date{\today}

\begin{abstract}
We introduce a ``Broken-Arrow'' matrix model for the $\beta$-Wishart ensemble, which improves on the traditional bidiagonal model by generalizing to non-identity covariance parameters. We prove that its joint eigenvalue density involves the correct hypergeometric function of two matrix arguments, and a continuous parameter $\beta > 0$.

If we choose $\beta = 1,2,4$, we recover the classical Wishart ensembles of general covariance over the reals, complexes, and quaternions. The derivation of the joint density requires an interesting new identity about Jack polynomials in $n$ variables. Jack polynomials are often defined as the eigenfunctions of the Laplace-Beltrami Operator. We prove that Jack polynomials are in addition eigenfunctions of an integral operator defined as an average over a $\beta$-dependent measure on the sphere. When combined with an identity due to Stanley, \cite{Stanley1989} we derive a new definition of Jack polynomials.

An efficient numerical algorithm is also presented for simulations. The algorithm makes use of secular equation software for broken arrow matrices currently unavailable in the popular technical computing languages. The simulations are matched against the cdf's for the extreme eigenvalues.

The techniques here suggest that arrow and broken arrow matrices  can play an important role in theoretical and computational random matrix theory including the study of corners processes.
\end{abstract}

\maketitle

\newtheorem{thm}{Theorem}
\newtheorem{lem}{Lemma}
\newtheorem{cor}{Corollary}
\newtheorem{definition}{Definition}

\section{Introduction}
\label{Introduction}

A real $m\times n$ Wishart matrix $W(D,m,n)$ is the random matrix $Z^tZ$ where $Z$ consists of $n$ columns of length $m$, each of which is a multivariate normal with mean $0$ and covariance $D$. We can assume without loss of generality that $D$ is a non-negative diagonal matrix. We may write $Z$ as \verb randn(m,n)*sqrt(D)  using the notation of modern technical computing software. Real Wishart matrices arise in such applications as likelihood-ratio tests (summarized in Ch. 8 of Muirhead\cite{Muirhead1982}), multidimensional Bayesian analysis\cite{Bekker1995},\cite{Evans1964}, and random matrix theory in general\cite{James}.

For the special case that $D = I$, the real Wishart matrix is also known as the $\beta=1$ Laguerre ensemble. The complex and quaternion versions correspond to $\beta=2$ and $\beta=4$, respectively. The method of bidiagonalization has been very successful \cite{Dumitriu2002} in creating matrix models that generalize the Laguerre ensemble to arbitrary $\beta$. The key features of the Dumitriu and Edelman model\cite{Dumitriu2002} appear in the box below: \vskip .1in

\noindent\fbox{\begin{minipage}{6.2in}
\begin{center}
\underline{Beta-Laguerre Model} \\ \vskip .1in
$B = \left[\begin{array}{ccccc}
\chi_{m\beta} & \chi_{(n-1)\beta} & & & \\
& \chi_{(m-1)\beta} & \chi_{(n-2)\beta} & & \\
& & \ddots & \ddots & \\
& & & \chi_{(m-n+2)\beta} & \chi_{\beta} \\
& & & & \chi_{(m-n+1)\beta}
\end{array}\right]$ \\
\end{center}
Let $n$ be a positive integer and $m$ be a real greater than $n-1$. ${\rm eig}(B^tB)$ has density:
\begin{center}
$C_\beta^{L}\prod_{i=1}^{n}\lambda_i^{\frac{m-n+1}{2}\beta-1}\prod_{i<j}|\lambda_i-\lambda_j|^{\beta}{}\exp(-\frac{1}{2}\sum_{i=1}^{n}\lambda_i)d\lambda$\\ \vskip .1in
\end{center}
\end{minipage}
} \vskip .1in

For general $D$, it is desirable to create a general $\beta$ model as well. For $\beta = 1,2,4$, it is obviously possible to bidiagonalize a full matrix of real, complex, or quaternionic normals and obtain a real bidiagonal matrix. However these bidiagonal models do not appear to generalize nicely to arbitrary $\beta$. Therefore we propose a new broken-arrow model that possesses a number of very good mathematical (and computational) possibilities. These models connect to the mathematical theory of Jack polynomials and the general-$\beta$ hypergeometric functions of matrix arguments. The key features of the broken arrow models appear in the box below:

\noindent\fbox{\begin{minipage}{6.2in}
\begin{center}
\underline{Beta-Wishart (Recursive) Model, $W^{(\beta)}(D,m,n)$} \\ \vskip .1in
$Z = \left[\begin{array}{cccr}
\tau_1 &            &                   & \chi_{\beta}D_{n,n}^{1/2} \\
           & \ddots &                  & \vdots \\
           &            & \tau_{n-1} & \chi_{\beta}D_{n,n}^{1/2} \\
           &            &                  & \chi_{(m-n+1)\beta}D_{n,n}^{1/2}
\end{array}\right],$ \\ \vskip .1in
\end{center}
where $\{\tau_1,\ldots,\tau_{n-1}\}$ are the singular values of $W^{(\beta)}(D_{1:n-1,1:n-1},m,n-1)$, base case $W(D,m,1)$ is $\tau_1 = \chi_{m\beta}D_{1,1}^{1/2}$. Let $n$ be a positive integer and $m$ a real greater than $n-1$. Let $\Lambda = {\rm diag}(\lambda_1,\ldots,\lambda_n)$, ${\rm eig}(Z^tZ)$ has density:\begin{center}
$C_\beta^{W}\det(D)^{-m\beta/2}\prod_{i=1}^{n}\lambda_i^{\frac{m-n+1}{2}\beta-1}\Delta(\lambda)^{\beta}\cdot {{}_0F_0}^{(\beta)}\left(-\frac{1}{2}\Lambda,D^{-1}\right)d\lambda.$\\ \end{center} \vskip .1in
\end{minipage}
} \vskip .1in

Theorem 3 proves that ${\rm eig}(Z^tZ)$ is distributed by the formula above. This generalizes the work on the special cases ($\beta = 1$)\cite{James}, ($\beta = 2$)\cite{Ratnarajah2004}, and ($\beta = 4$)\cite{Li2009}  which found the eigenvalue distributions for full-matrix Wishart ensembles for their respective $\beta$'s. While the model appears straightforward, there are a number of hidden mathematical and computational complexities. The proof of Theorem 3, that the broken arrow model has the joint eigenvalue density in the box above, relies on a theorem about Jack polynomials which appears to be new (Corollary 1 to Theorem 2):

\noindent\textbf{Corollary 1.} \textit{Let $C_\kappa^{(\beta)}(\Lambda)$ be the Jack polynomial under the $C$-normalization and $dq$ be the surface area element on the positive quadrant of the unit sphere. Then}
$$ C_\kappa^{(\beta)}(\Lambda) = 
\frac{C_\kappa^{(\beta)}(I_n)}{C_\kappa^{(\beta)}(I_{n-1})}\cdot\frac{2^{n-1}\Gamma(n\beta/2)}{\Gamma(\beta/2)^n}\cdot\int\prod_{i=1}^{n}q_i^{\beta-1}C_\kappa^{(\beta)}((I-qq^t)\Lambda)dq. $$
\textit{Equivalently,}
$$ C_\kappa^{(\beta)}(\Lambda) \propto E_q(C_\kappa^{(\beta)}({\rm Proj}_{q^{\perp}}\Lambda)), $$
\textit{where the expectation is taken over length $n$ vectors of $\chi_\beta$'s renormalized to lie on the unit sphere.}

Given a unit vector $q$, one can create a projected Jack polynomial of a symmetric matrix by projecting out the direction $q$. One can reconstruct the Jack polynomial by averaging over all directions $q$. This process is reminiscent of integral geometry \cite{Santalo1976} or computerized tomography. For $\beta = 1$, the measure on the sphere is the uniform measure, for other $\beta$'s the measure on the sphere is proportional to $\prod_{i=1}^{n}q_i^{\beta-1}$. In addition, using the above formula with a fact from Stanley\cite{Stanley1989}, we derive a new definition of the Jack polynomials (see Corollary 2).

\noindent\textbf{Remark:} Since deriving Corollary 1, we learned that it may be obtained  from Okounkov and Olshanski\cite{Okounkov1997} (Proposition on p. 8) through a change of variables.  We found Corollary 1 on our own, as we needed it for our matrix model to give the correct answer. We thank Alexei Borodin for interesting discussions to help us see the connection. An interesting lesson we learned is that Jack Polynomials with matrix argument notation can reveal important relationships that multivariate argument notation can hide.  

Inspiration for this broken arrow model came from the mathematical method of ghosts and shadows \cite{Edelman2009} and numerical techniques for the svd updating problem. Algorithms for updating the svd one column at a time were first considered by Bunch and and Nielson\cite{Bunch1978}. Gu and Eisenstat\cite{Gu1994} present an improvement using the fast multipole method. The random matrix context here is simpler than the numerical situation in that orthogonal invariance replaces the need for singular vector updating.

Software for efficiently computing the svd of broken arrow matrices is unavailable in the currently popular technical computing languages such as MATLAB, Mathematica, Maple, R, or Python. The sophisticated secular equation solver, LAPACK's dlasd4.f, efficiently computes the singular values of a broken arrow matrix. Using this software, we can sample the eigenvalues of $W^{(\beta)}(D,m,n)$ in $O(n^3)$ time and $O(n)$ space.

In this paper we perform a number of numerical simulations as well to confirm the correctness and illustrate applications of the model.  Among these simulations are largest and smallest eigenvalue densities. We also use free probability to histogram the eigenvalues of $W^{(\beta)}(D,m,n)$ for general $\beta,m,n$ and $D$ drawn from a prior, and show that they match the analytical predictions of free probability made by Olver and Nadakuditi\cite{Olver2013}.

\section{Real, Complex, Quaternion, and Ghost Wishart Matrices}

Let $G_\beta$ represent a Gaussian real, complex, or quaternion for $\beta = 1,2,4$, with mean zero and variance one. Let $\chi_d$ be a $\chi$-distributed real with $d$ degrees of freedom. The following algorithm computes the singular values, where all of the random variables in a given matrix are assumed independent. We assume $D=I$ for purposes of illustration, but this algorithm generalizes. We proceed through a series of matrices related by orthogonal transformations on the left and the right.
$$ 
\left[\begin{array}{ccc}
G_\beta & G_\beta & G_\beta \\
G_\beta & G_\beta & G_\beta \\
G_\beta & G_\beta & G_\beta
\end{array}\right]
\longrightarrow
\left[\begin{array}{ccc}
\chi_{3\beta} & G_\beta & G_\beta \\
0 & G_\beta & G_\beta \\
0 & G_\beta & G_\beta
\end{array}\right]
\longrightarrow
\left[\begin{array}{ccc}
\chi_{3\beta} & \chi_\beta & G_\beta \\
0 & \chi_{2\beta} & G_\beta \\
0 & 0 & G_\beta
\end{array}\right]
$$
To create the real, positive $(1,2)$ entry, we multiply the second column by a real sign, or a complex or quaternionic phase. We then use a Householder reflector on the bottom two rows to make the $(2,2)$ entry a $\chi_{2\beta}$. Now we take the SVD of the $2\times 2$ upper-left block:
$$
\left[\begin{array}{ccc}
\tau_1 & 0 & G_\beta \\
0 & \tau_2 & G_\beta \\
0 & 0 & G_\beta
\end{array}\right]
\longrightarrow
\left[\begin{array}{ccc}
\tau_1 & 0 & \chi_\beta \\
0 & \tau_2 & \chi_\beta \\
0 & 0 & \chi_\beta
\end{array}\right]
\longrightarrow
\left[\begin{array}{ccc}
\sigma_1 & 0 & 0 \\
0 & \sigma_2 & 0 \\
0 & 0 & \sigma_3
\end{array}\right]
$$
We convert the third column to reals using a diagonal matrix of signs on both sides. The process can be continued for a larger matrix, and can work with one that is taller than is wide. What it proves is that the second-to-last-matrix,
$$\left[\begin{array}{ccc}
\tau_1 & 0 & \chi_\beta \\
0 & \tau_2 & \chi_\beta \\
0 & 0 & \chi_\beta
\end{array}\right],$$
has the same singular values as the first matrix, if $\beta = 1,2,4$. We call this new matrix a ``Broken-Arrow Matrix.'' It is reasonable to conjecture that such a procedure might work for all $\beta > 0$ in a way not yet defined. This idea is the basis of the method of ghosts and shadows \cite{Edelman2009}. We prove that for a general broken arrow matrix model, the singular value distribution is what the method of ghosts and shadows predicts for a $\beta$-dimensional algebra.

The following algorithm, which generalizes the one above for the $3\times 3$ case, samples the singular values of the Wishart ensemble for general $\beta$ and general $D$.

\fbox{
\begin{minipage}{5.6 in}
\begin{center}
\underline{Beta-Wishart (Recursive) Model Pseudocode}
\end{center}
\begin{algorithmic}
\State \textbf{Function} $\Sigma$ := BetaWishart$(m,n,\beta,D)$
\If{$n = 1$}
	\State $\Sigma$ := $\chi_{m\beta}D_{1,1}^{1/2}$
\Else
	\State $Z_{1:n-1,1:n-1}$ := BetaWishart$(m,n-1,\beta,D_{1:n-1,1:n-1})$
	\State $Z_{n,1:n-1}$ := $[0,\ldots,0]$
	\State $Z_{1:n-1,n}$ := $[\chi_\beta D_{n,n}^{1/2};\ldots;\chi_{\beta}D_{n,n}^{1/2}]$
	\State $Z_{n,n}$ := $\chi_{(m-n+1)\beta}D_{n,n}^{1/2}$
	\State $\Sigma$ := diag(svd($Z$))
\EndIf
\end{algorithmic}
\end{minipage}
} \vskip .1 in

The diagonal of $\Sigma$ contains the singular values. Since we know the distribution of the singular values of such a full matrix for $(\beta = 1,2,4)$, \cite{James}, \cite{Ratnarajah2004}, \cite{Li2009}, we can state (using the normalization constant in Corollary 3, originally from Forrester\cite{Forrester1994}):

\begin{thm}
The distribution of the singular values ${\rm diag}(\Sigma) = (\sigma_1,\ldots,\sigma_n)$, $\sigma_1>\sigma_2>\cdots>\sigma_n$, generated by the above algorithm for $\beta = 1,2,4$ is equal to:
$$ \frac{2^{n}\det(D)^{-m\beta/2}}{\mathcal{K}_{m,n}^{(\beta)}}\prod_{i=1}^{n}\sigma_i^{(m-n+1)\beta-1}\Delta^2(\sigma)^\beta {{{}_0F_0}}^{\!\!(\beta)}\left(-\frac{1}{2}\Sigma^2,D^{-1}\right)d\sigma, $$
where
$$ \mathcal{K}^{(\beta)}_{m,n} = \frac{2^{mn\beta/2}}{\pi^{n(n-1)\beta/2}}\cdot\frac{\Gamma_n^{(\beta)}(m\beta/2)\Gamma_n^{(\beta)}(n\beta/2)}{\Gamma(\beta/2)^n}, $$
and the generalized Gamma function $\Gamma_n^{(\beta)}$ is defined in Definition 6.
\end{thm}

Theorem 3 generalizes Theorem 1 to the $\beta > 0$ case. Before we can prove Theorem 3, we need some background.

\section{Arrow and Broken-Arrow Matrix Jacobians}

Define the (symmetric) Arrow Matrix
$$
A = \left[\begin{array}{cccc}
d_1 &              &               & c_1 \\
       & \ddots   &               & \vdots \\
       &              & d_{n-1} & c_{n-1} \\
c_1 & \cdots   & c_{n-1} & c_n
\end{array}\right].
$$
Let its eigenvalues be $\lambda_1,\ldots,\lambda_n$. Let $q$ be the last row of its eigenvector matrix, i.e. $q$ contains the $n$-th element of each eigenvector. $q$ is by convention in the positive quadrant.

Define the broken arrow matrix $B$ by
$$ B =
\left[\begin{array}{cccc}
b_1 &              &               & a_1 \\
       & \ddots   &               & \vdots \\
       &              & b_{n-1} & a_{n-1} \\
0 & \cdots   & 0 & a_n
\end{array}\right].
 $$
Let its singular values be $\sigma_1,\ldots,\sigma_n$, and let $q$ contain the bottom row of its right singular vector matrix, i.e. $A = B^tB$, $B^tB$ is an arrow matrix. $q$ is by convention in the positive quadrant.

Define $dq$ to be the surface-area element on the sphere in $\mathbb{R}^n$.

\begin{lem}
For an arrow matrix $A$, let $f$ be the unique map $f:(c,d)\longrightarrow(q,\lambda)$. The Jacobian of $f$ satisfies:
$$ dq d\lambda = \frac{\prod_{i=1}^{n}q_i}{\prod_{i=1}^{n-1}c_i}\cdot dcdd. $$
\end{lem}
\noindent The proof is after Lemma 3.

\begin{lem}
For a broken arrow matrix $B$, let $g$ be the unique map $g:(a,b)\longrightarrow(q,\sigma)$. The Jacobian of $g$ satisfies:
$$ dq d\sigma = \frac{\prod_{i=1}^{n}q_i}{\prod_{i=1}^{n-1}a_i}\cdot dadb. $$
\end{lem}
\noindent The proof is after Lemma 3.

\begin{lem}
If all elements of $a,b,q,\sigma$ are nonnegative, and $b,d,\lambda,\sigma$ are ordered, then $f$ and $g$ are bijections excepting sets of measure zero (if some $b_i = b_j$ or some $d_i = d_j$ for $i\neq j$).
\end{lem}
\begin{proof} We only prove it for $f$; the $g$ case is similar. We show that $f$ is a bijection using results from Dumitriu and Edelman\cite{Dumitriu2002}, who in turn cite Parlett\cite{Parlett1998}. Define the tridiagonal matrix
$$
\left[\begin{array}{ccccc}
\eta_1 & \epsilon_1 & 0 & 0 \\
\epsilon_1 & \eta_2 & \epsilon_2 & 0\\
 & \ddots & \ddots & \ddots  \\
0 & 0 & \epsilon_{n-1} & \eta_{n-1}
\end{array}
\right]
$$
to have eigenvalues $d_1,\ldots,d_{n-1}$ and bottom entries of the eigenvector matrix
$u = (c_1,\ldots,c_{n-1})/\gamma,$ where $\gamma = \sqrt{c_1^2+\cdots+c_{n-1}^2}$. Let the whole eigenvector matrix be $U$. $(d,u)\leftrightarrow(\epsilon,\eta)$ is a bijection\cite{Dumitriu2002},\cite{Parlett1998} excepting sets of measure $0$. Now we extend the above tridiagonal matrix further and use $\sim$ to indicate similar matrices:
$$
\left[\begin{array}{cccccc}
\eta_1 & \epsilon_1 & 0 & 0 & 0\\
\epsilon_1 & \eta_2 & \epsilon_2 & 0 & 0\\
 & \ddots & \ddots & \ddots & \\
0 & 0 & \epsilon_{n-1} & \eta_{n-1} & \gamma \\
0 & 0 & 0 & \gamma & c_n
\end{array}
\right] \sim
\left[\begin{array}{cccc}
d_1 &  &  & u_1\gamma \\
& \ddots & & \vdots \\
& & d_{n-1} & u_{n-1}\gamma \\
u_1\gamma & \cdots & u_{n-1}\gamma & c_n
\end{array}\right] = A
$$
$(c_1,\ldots,c_{n-1})\leftrightarrow (u,\gamma)$ is a bijection, as is $(c_n)\leftrightarrow (c_n)$, so we have constructed a bijection from $(c_1,\ldots,c_{n-1},c_n,d_1,\ldots,d_{n-1})\leftrightarrow(c_n,\gamma,\eta,\epsilon)$, excepting sets of measure $0$. $(c_n,\gamma,\eta,\epsilon)$ defines a tridiagonal matrix which is in bijection with $(q,\lambda)$\cite{Dumitriu2002},\cite{Parlett1998}. Hence we have bijected $(c,d)\leftrightarrow (q,\lambda)$. The proof that $f$ is a bijection is complete.
\end{proof}

\noindent\textit{Proof of Lemma 1.} By Dumitriu and Edelman\cite{Dumitriu2002}, Lemma 2.9,
$$ dqd\lambda =  \frac{\prod_{i=1}^{n}q_i}{\gamma\prod_{i=1}^{n-1}\epsilon_i}dc_n d\gamma d\epsilon d\eta. $$
Also by Dumitriu and Edelman\cite{Dumitriu2002}, Lemma 2.9,
$$ dddu = \frac{\prod_{i=1}^{n-1}u_i}{\prod_{i=1}^{n-1}\epsilon_i} d\epsilon d\eta. $$
Together,
$$ dq d\lambda = \frac{\prod_{i=1}^{n}q_i}{\gamma\prod_{i=1}^{n-1}u_i}dc_n dd du d\gamma $$
The full spherical element is, using $\gamma$ as the radius,
$$ dc_1\cdots dc_{n-1} = \gamma^{n-2} du d\gamma. $$
Hence,
$$ dqd\lambda = \frac{\prod_{i=1}^{n}q_i}{\gamma^{n-1}\prod_{i=1}^{n-1}u_i}dcdd, $$
which by substitution is
$$ dqd\lambda = \frac{\prod_{i=1}^{n}q_i}{\prod_{i=1}^{n-1}c_i}dcdd $$

\noindent\textit{Proof of Lemma 2.} Let $A = B^tB$. $d\lambda = 2^n\prod_{i=1}^{n}\sigma_i d\sigma$, and since $\prod_{i=1}^{n}{\sigma_i}^2 = \det(B^tB) = \det(B)^2 = a_n^2\prod_{i=1}^{n-1}b_i^2$, by Lemma 1,
$$ dqd\sigma = \frac{\prod_{i=1}^{n}q_i}{2^na_n\prod_{i=1}^{n-1}(b_ic_i)}dcdd. $$
The full-matrix Jacobian $\frac{\partial(c,d)}{\partial(a,b)}$ is
$$ \frac{\partial(c,d)}{\partial(a,b)} =
\left[\begin{array}{ccccccc}
b_1 & &        & 2a_1         & & & \\
 & \ddots &  & \vdots     & & & \\
& & b_{n-1} & 2a_{n-1}   & & & \\
 &  &            & 2a_n      & & & \\
 a_1  &      &       &         & 2b_1 & & \\
 & \ddots & &                & & \ddots & \\
 & & a_{n-1} &                 & & & 2b_{n-1} \\       
\end{array}
\right]
$$
The determinant gives $dcdd = 2^na_n\prod_{i=1}^{n-1}b_i^2 dadb$. So,
$$ dqd\sigma = \frac{\prod_{i=1}^{n}q_i\prod_{i=1}^{n-1}b_i}{\prod_{i=1}^{n-1}c_i}dadb = \frac{\prod_{i=1}^{n}q_i}{\prod_{i=1}^{n-1}a_i}dadb. $$

\section{Further Arrow and Broken-Arrow Matrix Lemmas}

\begin{lem}
$$ q_k = \left(1+\sum_{j=1}^{n-1}\frac{c_j^2}{(\lambda_k-d_j)^2}\right)^{-1/2}. $$
\end{lem}
\begin{proof}
Let $v$ be the eigenvector of $A$ corresponding to $\lambda_k$. Temporarily fix $v_n = 1$. Using $Av = \lambda v$, for $j < n$, $v_j = c_j/(\lambda_k - d_j)$. Renormalizing $v$ so that $\|v\| = 1$, we get the desired value for $v_n = q_k$.
\end{proof}

\begin{lem}
For a vector $x$ of length $l$, define $\Delta(x) = \prod_{i<j}|x_i-x_j|$. Then,
$$ \Delta(\lambda) = \Delta(d)\prod_{k=1}^{n-1}|c_k|\prod_{k=1}^{n}q_k^{-1}. $$
\end{lem}
\begin{proof}
Using a result in Wilkinson\cite{Wilkinson1999}, the characteristic polynomial of $A$ is:
\begin{equation}
p(\lambda) = \prod_{i=1}^{n}(\lambda_i-\lambda) = \prod_{i=1}^{n-1}(d_i-\lambda)\left(c_n-\lambda-\sum_{j=1}^{n-1}\frac{c_j^2}{d_j-\lambda}\right)
\end{equation}
Therefore, for $k < n$,
\begin{equation}
p(d_k) = \prod_{i=1}^{n}(\lambda_i-d_k) = -c_k^2\prod_{i=1,i\neq k}^{n-1}(d_i-d_k). 
\end{equation}
Taking a product on both sides,
$$ \prod_{i=1}^{n}\prod_{k=1}^{n-1}(\lambda_i-d_k) = (-1)^{n-1}\Delta(d)^2\prod_{k=1}^{n-1}c_k^2. $$
Also,
\begin{equation} p'(\lambda_k) = -\prod_{i=1,i\neq k}^{n}(\lambda_i-\lambda_k) = -\prod_{i=1}^{n-1}(d_i-\lambda_k)\left(1+\sum_{j=1}^{n-1}\frac{c_j^2}{(d_j-\lambda_k)^2}\right). \end{equation}
Taking a product on both sides,
$$ \prod_{i=1}^{n}\prod_{k=1}^{n-1}(\lambda_i-d_k) = (-1)^{n-1}\Delta(\lambda)^2\prod_{i=1}^{n}\left(1+\sum_{j=1}^{n-1}\frac{c_j^2}{(d_j-\lambda_i)^2}\right)^{-1}. $$
Equating expressions equal to $\prod_{i=1}^{n}\prod_{k=1}^{n-1}(\lambda_i-d_k)$, we get
$$ \Delta(d)^2\prod_{k=1}^{n-1}c_k^2 = \Delta(\lambda)^2\prod_{i=1}^{n}\left(1+\sum_{j=1}^{n-1}\frac{c_j^2}{(d_j-\lambda_i)^2}\right)^{-1}. $$
The desired result follows by the previous lemma.
\end{proof}

\begin{lem}
For a vector $x$ of length $l$, define $\Delta^2(x) = \prod_{i<j}|x_i^2-x_j^2|$. The singular values of $B$ satisfy
$$ \Delta^2(\sigma) = \Delta^2(b)\prod_{k=1}^{n-1}|a_k b_k|\prod_{k=1}^{n}q_k^{-1}. $$
\end{lem}
\begin{proof}
Follows from $A = B^tB.$
\end{proof}

\section{Jack and Hermite Polynomials}

As in \cite{Dumitriu2007}, if $\kappa\vdash k$, $\kappa = (\kappa_1,\kappa_2,\ldots)$ is nonnegative, ordered non-increasingly, and it sums to $k$. Let $\alpha = 2/\beta$. Let $\rho_\kappa^{\alpha} = \sum_{i=1}^{l}\kappa_i(\kappa_i-1-(2/\alpha)(i-1))$.
We define $l(\kappa)$ to be the number of nonzero elements of $\kappa$. We say that $\mu \leq \kappa$ in ``lexicographic ordering'' if for the largest integer $j$ such that $\mu_i = \kappa_i$ for all $i < j$, we have $\mu_j \leq \kappa_j$.

\begin{definition}
As in Dumitriu, Edelman and Shuman\cite{Dumitriu2007}, we define the Jack polynomial of a matrix argument, $C^{(\beta)}_\kappa(X)$, as follows: Let $x_1,\ldots,x_n$ be the eigenvalues of $X$. $C^{(\beta)}_\kappa(X)$ is the only homogeneous polynomial eigenfunction of the Laplace-Beltrami-type operator:
$$ D^*_n = \sum_{i=1}^{n}x_i^2\frac{\partial^2}{\partial x_i^2} + \beta\cdot\sum_{1\leq i\neq j \leq n}\frac{x_i^2}{x_i-x_j}\cdot\frac{\partial}{\partial x_i}, $$
with eigenvalue $\rho_k^{\alpha}+k(n-1),$ having highest order monomial basis function in lexicographic ordering (see \cite{Dumitriu2007}, Section 2.4) corresponding to $\kappa$. In addition,
$$ \sum_{\kappa\vdash k,l(\kappa)\leq n} C_{\kappa}^{(\beta)}(X) = {\rm trace}(X)^k.$$
\end{definition}

\begin{lem}
If we write $C^{(\beta)}_\kappa(X)$ in terms of the eigenvalues $x_1,\ldots,x_n$, as $C^{(\beta)}_\kappa(x_1,\ldots,x_n)$, then $C^{(\beta)}_\kappa(x_1,\ldots,x_{n-1}) = C^{(\beta)}_\kappa(x_1,\ldots,x_{n-1},0)$ if $l(\kappa) < n$. If $l(\kappa) = n$, $C^{(\beta)}_\kappa(x_1,\ldots,x_{n-1},0) = 0$.
\end{lem}
\begin{proof}
The $l(\kappa) = n$ case follows from a formula in Stanley\cite{Stanley1989}, Propositions 5.1 and 5.5 that only applies if $\kappa_n > 0$,
$$ C_\kappa^{(\beta)}(X) \propto \det(X)C_{\kappa_1-1,\ldots,\kappa_n-1}^{(\beta)}(X). $$
If $\kappa_n = 0$, from Koev\cite[(3.8)]{Koev2006}, $C^{(\beta)}_\kappa(x_1,\ldots,x_{n-1}) = C^{(\beta)}_\kappa(x_1,\ldots,x_{n-1},0)$.
\end{proof}

\begin{definition}
The Hermite Polynomials (of a matrix argument) are a basis for the space of symmetric multivariate polynomials over eigenvalues $x_1,\ldots,x_n$ of $X$ which are related to the Jack polynomials by (Dumitriu, Edelman, and Shuman\cite{Dumitriu2007}, page 17)
$$ H_\kappa^{(\beta)}(X) = \sum_{\sigma\subseteq\kappa}c^{(\beta)}_{\kappa,\sigma}\cdot\frac{C_{\sigma}^{(\beta)}(X)}{C_{\sigma}^{(\beta)}(I_n)}, $$
where $\sigma\subseteq\kappa$ means for each $i$, $\sigma_i\leq\kappa_i$, and the coefficicents $c_{\kappa,\sigma}^{(\beta)}$ are given by (Dumitriu, Edelman, and Shuman\cite{Dumitriu2007}, page 17). Since Jack polynomials are homogeneous, that means 
$$ H_\kappa^{(\beta)}(X) \propto C_\kappa^{(\beta)}(X) + L.O.T.  $$
Furthermore, by (Dumitriu, Edelman, and Shuman\cite{Dumitriu2007}, page 16), the Hermite Polynomials are orthogonal with respect to the measure
$$ \exp\left(-\frac{1}{2}\sum_{i=1}^{n}x_i^2\right)\prod_{i\neq j}|x_i-x_j|^{\beta}. $$
\end{definition}

\begin{lem}
Let
$$
A(\mu,c) = \left[\begin{array}{cccc}
\mu_1 &              &               & c_1 \\
       & \ddots   &               & \vdots \\
       &              & \mu_{n-1} & c_{n-1} \\
c_1 & \cdots   & c_{n-1} & c_n
\end{array}\right]
=
\left[\begin{array}{cccc}
 &              &               & c_1 \\
       & M   &               & \vdots \\
       &              &   & c_{n-1} \\
c_1 & \cdots   & c_{n-1} & c_n
\end{array}\right],
 $$
and let for $l(\kappa) < n$,
$$ Q(\mu,c_n) = \int\prod_{i=1}^{n-1}c_i^{\beta-1} H_{\kappa}^{(\beta)}(A(\mu,c)) \exp(-c_1^2-\cdots-c_{n-1}^2)dc_1\cdots dc_{n-1}. $$
$Q$ is a symmetric polynomial in $\mu$ with leading term proportional to $H_\kappa^{(\beta)}(M)$ plus terms of order strictly less than $|\kappa|$.
\end{lem}
\begin{proof}
If we exchange two $c_i$'s, $i<n$, and the corresponding $\mu_i$'s, $A(\mu,c)$ has the same eigenvalues, so $H_\kappa^{(\beta)}(A(\mu,c))$ is unchanged. So, we can prove $Q(\mu,c_n)$ is symmetric in $\mu$ by swapping two $\mu_i$'s, and seeing that the integral is invariant over swapping the corresponding $c_i$'s.

Now since $H_\kappa^{(\beta)}(A(\mu,c))$ is a symmetric polynomial in the eigenvalues of $A(\mu,c)$, we can write it in the power-sum basis, i.e. it is in the ring generated by $t_p = \lambda_1^p + \cdots + \lambda_n^p$, for $p = 0, 1, 2, 3, \ldots$, if $\lambda_1,\ldots,\lambda_n$ are the eigenvalues of $A(\mu,c)$. But $t_p = {\rm trace}(A(\mu,c)^p)$, so it is a polynomial in $\mu$ and $c$,
$$ H_\kappa^{(\beta)}(A(\mu,c)) = \sum_{i\geq 0}\sum_{\epsilon_1,\ldots,\epsilon_{n-1}\geq 0}p_{i,\epsilon}(\mu)c_{n}^ic_1^{\epsilon_1}\cdots c_{n-1}^{\epsilon_{n-1}}. $$
Its order in $\mu$ and $c$ must be $|\kappa|$, the same as its order in $\lambda$. Integrating, it follows that
$$ Q(\mu,c_n) = \sum_{i\geq 0}\sum_{\epsilon_1,\ldots,\epsilon_{n-1}\geq 0}p_{i,\epsilon}(\mu)c_{n}^iM_{\epsilon}, $$
for constants $M_{\epsilon}$. Since $\deg(H_\kappa^{(\beta)}(A(\mu,c))) = |\kappa|$, $\deg(p_{i,\epsilon}(\mu)) \leq |\kappa|-|\epsilon|-i$. Writing
$$ Q(\mu,c_n) = M_{\vec{0}}p_{0,\vec{0}}(\mu) + \sum_{(i,\epsilon)\neq(0,\vec{0})}p_{i,\epsilon}(\mu)c_n^{i}M_\epsilon, $$
we see that the summation has degree at most $|\kappa|-1$ in $\mu$ only, treating $c_n$ as a constant. Now
$$ p_{0,\vec{0}}(\mu) = H_\kappa^{(\beta)}\left(\left[\begin{array}{cc}  
M & \vec{0} \\
\vec{0} & 0
\end{array}\right]\right) = H_{\kappa}^{(\beta)}(\mu) + r(\mu), $$
where $r(\mu)$ has degree at most $|\kappa|-1$. This follows from the expansion of $H_{\kappa}^{(\beta)}$ in Jack polynomials in Definition 2 and the fact about Jack polynomials in Lemma 7. The new lemma follows.
\end{proof}

\begin{lem}
Let the arrow matrix below have eigenvalues in $\Lambda = diag(\lambda_1,\ldots,\lambda_n)$ and have $q$ be the last row of its eigenvector matrix, i.e. $q$ contains the $n$-th element of each eigenvector,
$$
A(\Lambda,q) = \left[\begin{array}{cccc}
\mu_1 &              &               & c_1 \\
       & \ddots   &               & \vdots \\
       &              & \mu_{n-1} & c_{n-1} \\
c_1 & \cdots   & c_{n-1} & c_n
\end{array}\right]
=
\left[\begin{array}{cccc}
 &              &               & c_1 \\
       & M   &               & \vdots \\
       &              &   & c_{n-1} \\
c_1 & \cdots   & c_{n-1} & c_n
\end{array}\right],
$$
By Lemma 3 this is a well-defined map except on a set of measure zero. Then, for $U(X)$ a symmetric homogeneous polynomial of degree $k$ in the eigenvalues of $X$,
$$ V(\Lambda) = \int\prod_{i=1}^{n}q_i^{\beta-1}U(M)dq $$
is a symmetric homogeneous polynomial of degree $k$ in $\lambda_1,\ldots,\lambda_n$.
\end{lem}
\begin{proof}
Let $e_n$ be the column vector that is $0$ everywhere except in the last entry, which is $1$. $(I - e_ne_n^t)A(\Lambda,q)(I - e_ne_n^t)$ has eigenvalues $\{\mu_1,\ldots,\mu_{n-1},0\}$. If the eigenvector matrix of $A(\Lambda,q)$ is $Q$, so must
$$Q^t(I - e_ne_n^t)Q\Lambda Q^t(I - e_ne_n^t)Q$$
have those eigenvalues. But this is
$$ (I - qq^t)\Lambda(I-qq^t). $$
So
\begin{equation}
U(M) = U({\rm eig}((I - qq^t)\Lambda(I-qq^t))\backslash\{0\}).
\end{equation}
It is well known that we can write $U(M)$ in the power-sum ring, $U(M)$ is made of sums and products of functions of the form $\mu_1^p+\cdots+\mu_{n-1}^p$, where $p$ is a positive integer. Therefore, the RHS is made of functions of the form $$ \mu_1^p+\cdots+\mu_{n-1}^p+0^p = {\rm trace}(((I - qq^t)\Lambda(I-qq^t))^p), $$
which if $U(M)$ is order $k$ in the $\mu_i$'s, must be order $k$ in the $\lambda_i$'s. So $V(\Lambda)$ is a polynomial of order $k$ in the $\lambda_i's$. Switching $\lambda_1$ and $\lambda_2$ and also $q_1$ and $q_2$ leaves
$$ \int\prod_{i=1}^{n}q_i^{\beta-1}U({\rm eig}((I - qq^t)\Lambda(I-qq^t))\backslash\{0\}) dq $$
invariant, so $V(\Lambda)$ is symmetric.

\end{proof}

Theorem 2 is a new theorem about Jack polynomials.
\begin{thm}
Let the arrow matrix below have eigenvalues in $\Lambda = {\rm diag}(\lambda_1,\ldots,\lambda_n)$ and have $q$ be the last row of its eigenvector matrix, i.e. $q$ contains the $n$-th element of each eigenvector,
$$
A(\Lambda,q) = \left[\begin{array}{cccc}
\mu_1 &              &               & c_1 \\
       & \ddots   &               & \vdots \\
       &              & \mu_{n-1} & c_{n-1} \\
c_1 & \cdots   & c_{n-1} & c_n
\end{array}\right]
=
\left[\begin{array}{cccc}
 &              &               & c_1 \\
       & M   &               & \vdots \\
       &              &   & c_{n-1} \\
c_1 & \cdots   & c_{n-1} & c_n
\end{array}\right],
$$
By Lemma 3 this is a well-defined map except on a set of measure zero. Then, if for a partition $\kappa$, $l(\kappa) < n$, and $q$ on the first quadrant of the unit sphere,
$$ C^{(\beta)}_\kappa(\Lambda) \propto \int\prod_{i=1}^{n}q_i^{\beta-1}C_\kappa^{(\beta)}(M)dq. $$
\end{thm}
\begin{proof}
Define
$$ \eta_\kappa^{(\beta)}(\Lambda) = \int\prod_{i=1}^{n}q_i^{\beta-1}H_\kappa^{(\beta)}(M)dq. $$
This is a symmetric polynomial in $n$ variables (Lemma 9). Thus it can be expanded in Hermite polynomials with max order $|\kappa|$ (Lemma 9):
$$ \eta_\kappa^{(\beta)}(\Lambda) = \sum_{|\kappa^{(0)}|\leq|\kappa|} c(\kappa^{(0)},\kappa)H_{\kappa^{(0)}}^{(\beta)}(\Lambda), $$
where $|\kappa| = \kappa_1 + \kappa_2 + \cdots + \kappa_{l(\kappa)}$. Using orthogonality, from the previous definition of Hermite Polynomials,

$$ c(\kappa^{(0)},\kappa) \propto \int_{\Lambda\in\mathbb{R}^n}\int_{q}\prod_{i=1}^{n}q_i^{\beta-1}H_\kappa^{(\beta)}(M) H_{\kappa^{(0)}}^{(\beta)}(\Lambda) $$ $$\times \exp(-\frac{1}{2}{\rm trace}(\Lambda^2))\prod_{i\neq j}|\lambda_i-\lambda_{j}|^{\beta}dqd\lambda. $$
Using Lemmas 1 and 3,
$$ c(\kappa^{(0)},\kappa) \propto \int\prod_{i=1}^{n}q_i^{\beta-1}H_\kappa^{(\beta)}(M) H_{\kappa^{(0)}}^{(\beta)}(\Lambda) $$ $$\times \exp(-\frac{1}{2}{\rm trace}(\Lambda^2))\prod_{i\neq j}|\lambda_{i}-\lambda_{j}|^{\beta}\frac{\prod_{i=1}^{n}q_i}{\prod_{i=1}^{n-1}c_i}d\mu dc. $$
Using Lemma 6,
$$ c(\kappa^{(0)},\kappa) \propto \int\prod_{i=1}^{n-1}c_i^{\beta-1}H_\kappa^{(\beta)}(M) H_{\kappa^{(0)}}^{(\beta)}(\Lambda) $$ $$\times \exp(-\frac{1}{2}{\rm trace}(\Lambda^2))\prod_{i\neq j}|\mu_{i}-\mu_{j}|^{\beta}d\mu dc, $$
and by substitution
$$ c(\kappa^{(0)},\kappa) \propto \int\prod_{i=1}^{n-1}c_i^{\beta-1}H_\kappa^{(\beta)}(M) H_{\kappa^{(0)}}^{(\beta)}(A(\Lambda,q)) $$ $$\times \exp(-\frac{1}{2}{\rm trace}(A(\Lambda,q)^2))\prod_{i\neq j}|\mu_{i}-\mu_{j}|^{\beta}d\mu dc. $$
Define
$$ Q(\mu,c_n) = \int\prod_{i=1}^{n-1}c_i^{\beta-1} H_{\kappa^{(0)}}^{\beta}(A(\Lambda,q)) \exp(-c_1^2-\cdots-c_{n-1}^2)dc_1\cdots dc_{n-1}.$$
$Q(\mu,c_n)$ is a symmetric polynomial in $\mu$ (Lemma 8). Furthermore, by Lemma 8,
$$ Q(\mu,c_n) \propto H_{\kappa^{(0)}}^{(\beta)}(M) + L.O.T., $$
where the Lower Order Terms are of lower order than $|\kappa^{(0)}|$ and are symmetric polynomials. Hence they can be written in a basis of lower order Hermite Polynomials, and as
$$ c(\kappa^{(0)},\kappa) \propto \int H_\kappa^{(\beta)}(M)Q(\mu,c_n)$$$$\times\prod_{i\neq j}|\mu_i-\mu_j|^{\beta}\exp\left(-\frac{1}{2}(c_n^2 + \mu_1^2 + \cdots + \mu_{n-1}^2)\right) d\mu dc_{n}, $$
we have by orthogonality
$$ c(\kappa^{(0)},\kappa) \propto \delta(\kappa^{(0)},\kappa), $$
where $\delta$ is the Dirac Delta. So
$$ \eta_\kappa^{(\beta)}(\Lambda) = \int\prod_{i=1}^{n}q_i^{\beta-1}H_\kappa^{(\beta)}(M)dq \propto H_\kappa^{(\beta)}(\Lambda). $$
By Lemma 9, coupled with Definition 2,
$$ C_\kappa^{(\beta)}(\Lambda) \propto \int\prod_{i=1}^{n}q_i^{\beta-1}C_\kappa^{(\beta)}(M)dq. $$
\end{proof}

\begin{cor} Finding the proportionality constant: For $l(\kappa) < n$,
$$ C_\kappa^{(\beta)}(\Lambda) = 
\frac{C_\kappa^{(\beta)}(I_n)}{C_\kappa^{(\beta)}(I_{n-1})}\cdot\frac{2^{n-1}\Gamma(n\beta/2)}{\Gamma(\beta/2)^n}\cdot\int\prod_{i=1}^{n}q_i^{\beta-1}C_\kappa^{(\beta)}((I-qq^t)\Lambda)dq. $$
\end{cor}
\begin{proof}
By Theorem 2 with Equation (4) (in the proof of Lemma 9),
$$ C_\kappa^{(\beta)}(\Lambda) \propto \int\prod_{i=1}^{n}q_i^{\beta-1}C_\kappa^{(\beta)}({\rm eig}((I-qq^t)\Lambda(I-qq^t))-\{0\})dq, $$
which by Lemma 7 and properties of matrices is 
$$ C_\kappa^{(\beta)}(\Lambda) \propto \int\prod_{i=1}^{n}q_i^{\beta-1}C_\kappa^{(\beta)}((I-qq^t)\Lambda)dq. $$
Now to find the proportionality constant. Let $\Lambda = I_n$, and let $c_p$ be the constant of proportionality.
$$ C_\kappa^{(\beta)}(I_n) = c_p\cdot\int\prod_{i=1}^{n}q_i^{\beta-1}C_\kappa^{(\beta)}(I-qq^t)dq. $$
Since $I-qq^t$ is a projection, we can replace the term in the integral by $C_\kappa^{(\beta)}(I_{n-1})$, which can be moved out. So
$$ c_p = \frac{C_\kappa^{(\beta)}(I_n)}{C_\kappa^{(\beta)}(I_{n-1})}\left(\int\prod_{i=1}^{n}q_i^{\beta-1}dq\right)^{-1}. $$
Now
\begin{eqnarray*} \int\prod_{i=1}^{n}q_i^{\beta-1}dq &=& \frac{2}{\Gamma(n\beta/2)}\int_0^{\infty} r^{n\beta-1}e^{-r^2}dr\int\prod_{i=1}^{n}q_i^{\beta-1}dq \\
&=& \frac{2}{\Gamma(n\beta/2)}\int\int_{0}^{\infty}\prod_{i=1}^{n}(rq_i)^{\beta-1}e^{-r^2}(r^{n-1}drdq) \\
&=& \frac{2}{\Gamma(n\beta/2)}\int\prod_{i=1}^{n}x_i^{\beta-1}e^{-x_1^2-\cdots-x_n^2}dx \\
&=& \frac{2}{\Gamma(n\beta/2)}\left(\int_{0}^{\infty} x_1^{\beta-1}e^{-x_1^2}dx_1\right)^n \\
&=& \frac{2}{\Gamma(n\beta/2)}\left(\frac{\Gamma(\beta/2)}{2}\right)^n \\
&=& \frac{\Gamma(\beta/2)^n}{2^{n-1}\Gamma(n\beta/2)},
\end{eqnarray*}
and the corollary follows.
\end{proof}

\begin{cor}
The Jack polynomials can be defined recursively using Corollary 1 and two results in the compilation \cite{Koev2012}.
\end{cor}
\begin{proof}
By Stanley\cite{Stanley1989}, Proposition 4.2, the Jack polynomial of \textit{one} variable under the $J$ normalization is
$$ J_{\kappa_1}^{(\beta)}(\lambda_1) = \lambda_1^{\kappa_1}(1+(2/\beta))\cdots(1+(\kappa_1-1)(2/\beta)). $$
There exists another recursion for Jack polynomials under the $J$ normalization:
$$ J_\kappa^{(\beta)}(\Lambda) = \det(\Lambda)J_{(\kappa_1-1,\ldots,\kappa_{n}-1)}\prod_{i=1}^{n}(n-i+1+(2/\beta)(\kappa_i-1)), $$
if $\kappa_n > 0$. Note that if $\kappa_n > 0$ we can use the above formula to reduce the size of $\kappa$ in a recursive expression for a Jack polynomial, and if $\kappa_n = 0$ we can use Corollary 1 to reduce the number of variables in a recursive expression for a Jack polynomial. Using those facts together and the conversion between $C$ and $J$ normalizations in \cite{Dumitriu2007}, we can define all Jack polynomials.
\end{proof}

\section{Hypergeometric Functions}

\begin{definition}
We define the hypergeometric function of two matrix arguments and parameter $\beta$, ${{}_0F_0}^{\beta}(X,Y)$, for $n\times n$ matrices $X$ and $Y$, by
$$ {{}_0F_0}^{\beta}(X,Y) = \sum_{k=0}^{\infty}\sum_{\kappa\vdash k, l(\kappa)\leq n}\frac{C_\kappa^{(\beta)}(X)C_\kappa^{(\beta)}(Y)}{k!C_\kappa^{(\beta)}(I)}. $$
as in Koev and Edelman\cite{Koev2006}. It is efficiently calculated using the software described in Koev and Edelman\cite{Koev2006}, {\tt mhg}, which is available online\cite{PlamenURL}. The $C$'s are Jack polynomials under the $C$ normalization, $\kappa\vdash k$ means that $\kappa$ is a partition of the integer $k$, so $\kappa_1 \geq \kappa_2 \geq \cdots \geq 0$ have $|\kappa| = k = \kappa_1 + \kappa_2 + \cdots = k$.
\end{definition}

\begin{lem}
$$ {{}_0F_0}^{(\beta)}(X,Y) = \exp\left(s\cdot {\rm trace}(X)\right){{}_0F_0}^{(\beta)}(X,Y-sI). $$
\end{lem}
\begin{proof}
The claim holds for $s = 1$ by Baker and Forrester\cite{Baker1997}. Now, using that fact with the homogeneity of Jack polynomials,
$$ {{}_0F_0}^{(\beta)}(X,Y-sI) = {{}_0F_0}^{(\beta)}(X,s((1/s)Y-I)) = {{}_0F_0}^{(\beta)}(sX,(1/s)Y-I) $$
$$ = \exp\left(s\cdot {\rm trace}(X)\right){{}_0F_0}^{(\beta)}(sX,(1/s)Y) = \exp\left(s\cdot {\rm trace}(X)\right){{}_0F_0}^{(\beta)}(X,Y).$$
\end{proof}

\begin{definition}
We define the generalized Pochhammer symbol to be, for a partition $\kappa = (\kappa_1,\ldots,\kappa_l)$
$$ (a)_\kappa^{(\beta)} = \prod_{i=1}^{l}\prod_{j=1}^{\kappa_i}\left(a-\frac{i-1}{2}\beta+j-1\right). $$
\end{definition}

\begin{definition}
As in Koev and Edelman\cite{Koev2006}, we define the hypergeometric function ${}_1F_1$ to be
$$ {}_1F_1^{(\beta)}(a;b;X,Y) = \sum_{k=0}^{\infty}\sum_{\kappa\vdash k, l(\kappa)\leq n}\frac{(a)_\kappa^{\beta}}{(b)_\kappa^{(\beta)}}\cdot\frac{C_\kappa^{(\beta)}(X)C_\kappa^{(\beta)}(Y)}{k!C_\kappa^{(\beta)}(I)}. $$
The best software available to compute this function numerically is described in Koev and Edelman\cite{Koev2006}, {\tt mhg}.
\end{definition}

\begin{definition}
We define the generalized Gamma function to be
$$ \Gamma_n^{(\beta)}(c) = \pi^{n(n-1)\beta/4}\prod_{i=1}^{n}\Gamma(c-(i-1)\beta/2) $$
for $\Re(c) > (n-1)\beta/2$.
\end{definition}

\section{The $\beta$-Wishart ensemble, and its Spectral Distribution}

The $\beta$-Wishart ensemble for $m\times n$ matrices is defined iteratively; we derive the $m\times n$ case from the $m\times(n-1)$ case. 

\begin{definition} We assume $n$ is a positive integer and $m$ is a real greater than $n-1$. Let $D$ be a positive-definite diagonal $n\times n$ matrix. For $n = 1$, the $\beta$-Wishart ensemble is
$$
Z=\left[\begin{array}{c}
\chi_{m\beta}D_{1,1}^{1/2} \\
0 \\
\vdots \\
0
\end{array}\right],
$$
with $n-1$ zeros, where $\chi_{m\beta}$ represents a random positive real that is $\chi$-distributed with $m\beta$ degrees of freedom. For $n > 1$, the $\beta$-Wishart ensemble with positive-definite diagonal $n\times n$ covariance matrix $D$ is defined as follows: Let $\tau_1,\ldots,\tau_{n-1}$ be one draw of the singular values of the $m\times(n-1)$ $\beta$-Wishart ensemble with covariance $D_{1:(n-1),1:(n-1)}$. Define the matrix $Z$ by
$$
Z=
\left[\begin{array}{cccc}
\tau_1 & &        & \chi_{\beta}D_{n,n}^{1/2} \\
& \ddots &   & \vdots \\
& & \tau_{n-1} &  \chi_{\beta}D_{n,n}^{1/2} \\
& & &               \chi_{(m-n+1)\beta}D_{n,n}^{1/2}
\end{array}\right].
$$
All the $\chi$-distributed random variables are independent. Let $\sigma_1,\ldots,\sigma_n$ be the singular values of $Z$. They are one draw of the singular values of the $m\times n$ $\beta$-Wishart ensemble, completing the recursion. $\lambda_i = \sigma_i^2$ are the eigenvalues of the $\beta$-Wishart ensemble.
\end{definition}

\begin{thm}
Let $\Sigma = diag(\sigma_1,\ldots,\sigma_n)$, $\sigma_1>\sigma_2>\cdots>\sigma_n$. The singular values of the $\beta$-Wishart ensemble with covariance $D$ are distributed by a pdf proportional to
$$ \det(D)^{-m\beta/2}\prod_{i=1}^{n}\sigma_i^{(m-n+1)\beta-1}\Delta^2(\sigma)^\beta {{}_0F_0}^{(\beta)}\left(-\frac{1}{2}\Sigma^2,D^{-1}\right)d\sigma. $$
It follows from a simple change of variables that the ordered $\lambda_i$'s are distributed as
$$ C_{\beta}^{W}\det(D)^{-m\beta/2}\prod_{i=1}^{n}\lambda_i^{\frac{m-n+1}{2}\beta-1}\Delta(\lambda)^{\beta} {{}_0F_0}^{(\beta)}\left(-\frac{1}{2}\Lambda,D^{-1}\right)d\lambda. $$
\end{thm}

\begin{proof}
First we need to check the $n = 1$ case: the one singular value $\sigma_1$ is distributed as $\sigma_1 = \chi_{m\beta}D^{1/2}_{1,1}$, which has pdf proportional to
$$ D_{1,1}^{-m\beta/2}\sigma_1^{m\beta-1}\exp\left(-\frac{\sigma_1^2}{2D_{1,1}}\right)d\sigma_1. $$
We use the fact that
$${{}_0F_0}^{(\beta)}\left(-\frac{1}{2}\sigma_1^2,D_{1,1}^{-1}\right) = {{}_0F_0}^{(\beta)}\left(-\frac{1}{2D_{1,1}}\sigma_1^2,1\right) =  \exp\left(-\frac{\sigma_1^2}{2D_{1,1}}\right). $$
The first equality comes from the expansion of ${{}_0F_0}$ in terms of Jack polynomials and the fact that Jack polynomials are homogeneous, see the definition of Jack polynomials and ${{}_0F_0}$ in this paper, the second comes from (2.1) in Koev\cite{Koev2012}, or in Forrester\cite{Forrester2005}. We use that ${{}_0F_0}^{(\beta)}(X,I) = {{}_0F_0}^{(\beta)}(X),$ by definition\cite{Koev2006}.

Now we assume $n > 1$. Let 
$$
Z=
\left[\begin{array}{cccc}
\tau_1 & &        & \chi_{\beta}D_{n,n}^{1/2} \\
& \ddots &   & \vdots \\
& & \tau_{n-1} &  \chi_{\beta}D_{n,n}^{1/2} \\
& & &               \chi_{(m-n+1)\beta}D_{n,n}^{1/2}
\end{array}\right] =
\left[\begin{array}{cccc}
\tau_1 & &        & a_1 \\
& \ddots &   & \vdots \\
& & \tau_{n-1} & a_{n-1} \\
& & &                   a_n
\end{array}\right],
$$
so the $a_i$'s are $\chi$-distributed with different parameters. By hypothesis, the $\tau_i$'s are a $\beta$-Wishart draw. Therefore, the $a_i$'s and the $\tau_i$'s are assumed to have joint distribution proportional to
$$ \det(D)^{-m\beta/2}\prod_{i=1}^{n-1}\tau_i^{(m-n+2)\beta-1}\Delta^2(\tau)^{\beta}{{}_0F_0}^{(\beta)}\left(-\frac{1}{2}T^2,D_{1:n-1,1:n-1}^{-1}\right) $$ $$\times\left(\prod_{i=1}^{n-1}a_i\right)^{\beta-1} a_n^{(m-n+1)\beta-1}\exp\left(-\frac{1}{2D_{n,n}}\sum_{i=1}^{n}a_i^2\right)dad\tau, $$
where $T = {\rm diag}(\tau_1,\ldots,\tau_{n-1})$. Using Lemmas 2 and 3, we can change variables to
$$ \det(D)^{-m\beta/2}\prod_{i=1}^{n-1}\tau_i^{(m-n+2)\beta-1}\Delta^2(\tau)^{\beta}{{}_0F_0}^{(\beta)}\left(-\frac{1}{2}T^2,D_{1:n-1,1:n-1}^{-1}\right) $$ $$\times\left(\prod_{i=1}^{n-1}a_i\right)^{\beta} a_n^{(m-n+1)\beta-1}\exp\left(-\frac{1}{2D_{n,n}}\sum_{i=1}^{n}a_i^2\right)\prod_{i=1}^{n}q_i^{-1}\cdot d\sigma dq. $$
Using Lemma 6 this becomes:
$$
\det(D)^{-m\beta/2}\prod_{i=1}^{n-1}\tau_i^{(m-n+1)\beta-1}\Delta^2(\sigma)^{\beta} { }_0F_0^{(\beta)}(-\frac{1}{2}T^2,D_{1:n-1,1:n-1}^{-1})$$ $$\times\exp\left(-\frac{1}{2D_{n,n}}\sum_{i=1}^{n}a_i^2\right)\prod_{i=1}^{n}q_i^{\beta-1}d\sigma dq
$$
Using properties of determinants this becomes:
$$ \det(D)^{-m\beta/2}\prod_{i=1}^{n}\sigma_i^{(m-n+1)\beta-1}\Delta^2(\sigma)^{\beta}{{}_0F_0}^{(\beta)}\left(-\frac{1}{2}T^2,D_{1:n-1,1:n-1}^{-1}\right) $$ $$\times\exp\left(-\frac{1}{2D_{n,n}}\sum_{i=1}^{n}a_i^2\right)\prod_{i=1}^{n}q_i^{\beta-1}\cdot d\sigma dq. $$
To complete the induction, we need to prove
$$ {{}_0F_0}^{(\beta)}\left(-\frac{1}{2}\Sigma^2,D^{-1}\right) \propto \int\prod_{i=1}^{n}q_i^{\beta-1}e^{-\|a\|^2/(2D_{n,n})}{{}_0F_0}^{(\beta)}\left(-\frac{1}{2}T^2,D^{-1}_{1:(n-1),1:(n-1)}\right)dq. $$
We can reduce this expression using $\|a\|^2 + \sum_{i=1}^{n-1}\tau_i^2 = \sum_{i=1}^{n}\sigma_i^2$ that it suffices to show
$$ \exp\left({\rm trace}(\Sigma^2)/(2D_{n,n})\right){{}_0F_0}^{(\beta)}\left(-\frac{1}{2}\Sigma^2,D^{-1}\right)$$ $$ \propto \int\prod_{i=1}^{n}q_i^{\beta-1}\exp\left({\rm trace}(T^2)/(2D_{n,n})\right){{}_0F_0}^{(\beta)}\left(-\frac{1}{2}T^2,D^{-1}_{1:(n-1),1:(n-1)}\right)dq, $$
or moving some constants and signs around,
$$ \exp\left((-1/D_{n,n}){\rm trace}(-\Sigma^2/2)\right){{}_0F_0}^{(\beta)}\left(-\frac{1}{2}\Sigma^2,D^{-1}\right)$$ $$ \propto \int\prod_{i=1}^{n}q_i^{\beta-1}\exp\left((-1/(D_{n,n})){\rm trace}(-T^2/2)\right){{}_0F_0}^{(\beta)}\left(-\frac{1}{2}T^2,D^{-1}_{1:(n-1),1:(n-1)}\right)dq, $$
or using Lemma 10,
$${{}_0F_0}^{(\beta)}\left(-\frac{1}{2}\Sigma^2,D^{-1}-\frac{1}{D_{n,n}}I_n\right) $$ $$\propto \int\prod_{i=1}^{n}q_i^{\beta-1}{{}_0F_0}^{(\beta)}\left(-\frac{1}{2}T^2,D^{-1}_{1:(n-1),1:(n-1)}-\frac{1}{D_{n,n}}I_{n-1}\right)dq. $$
We will prove this expression termwise using the expansion of ${{}_0F_0}$ into infinitely many Jack polynomials. The $(k,\kappa)$ term on the right hand side is
$$ \int\prod_{i=1}^{n}q_i^{\beta-1}C_{\kappa}^{(\beta)}\left( -\frac{1}{2}T^2\right)C_{\kappa}^{(\beta)}\left(D_{1:(n-1),1:(n-1)}^{-1}-\frac{1}{D_{n,n}}I_{n-1}\right)dq, $$
where $\kappa\vdash k$ and $l(\kappa) < n$. The $(k,\kappa)$ term on the left hand side is
$$ C_{\kappa}^{(\beta)}\left( -\frac{1}{2}\Sigma^2\right)C_{\kappa}^{(\beta)}\left(D^{-1}-\frac{1}{D_{n,n}}I_{n-1}\right), $$
where $\kappa\vdash k$ and $l(\kappa) \leq n$. If $l(\kappa) = n$, the term is $0$ by Lemma 7, so either it has a corresponding term on the right hand side or it is zero. Hence, using Lemma 7 again it suffices to show that for $l(\kappa) < n$,
$$ C_{\kappa}^{(\beta)}(\Sigma^2) \propto \int\prod_{i=1}^{n}q_i^{\beta-1}C_{\kappa}^{(\beta)}(T^2)dq. $$
This follows by Theorem 2, and the proof of Theorem 3 is complete.
\end{proof}

\begin{cor}
The normalization constant, for $\lambda_1 > \lambda_2 > \cdots > \lambda_n$:
$$ C_\beta^{W} =  \frac{\det(D)^{-m\beta/2}}{\mathcal{K}_{m,n}^{(\beta)}}, $$
where
$$ \mathcal{K}^{(\beta)}_{m,n} = \frac{2^{mn\beta/2}}{\pi^{n(n-1)\beta/2}}\cdot\frac{\Gamma_n^{(\beta)}(m\beta/2)\Gamma_n^{(\beta)}(n\beta/2)}{\Gamma(\beta/2)^n}, $$
\end{cor}
\begin{proof}
We have used the convention that elements of $D$ do not move through $\propto$, so we may assume $D$ is the identity. Using ${{}_0F_0}^{(\beta)}(-\Lambda/2,I) = \exp\left(-{\rm trace}(\Lambda)/2\right)$, (Koev\cite{Koev2012}, (2.1)), the model becomes the $\beta$-Laguerre model studied in Forrester\cite{Forrester1994}.
\end{proof}

\begin{cor}
Using Definition 6 of the generalized Gamma, the distribution of $\lambda_{max}$ for the $\beta$-Wishart ensemble with general covariance in diagonal $D$, $P(\lambda_{max}<x)$, is:
$$ \frac{\Gamma_n^{(\beta)}(1+(n-1)\beta/2)}{\Gamma_n^{(\beta)}(1+(m+n-1)\beta/2)}
\det\left(\frac{x}{2} D^{-1}\right)^{m\beta/2}{}_1F_1^{(\beta)}\left(\frac{m}{2}\beta;\frac{m+n-1}{2}\beta+1;-\frac{x}{2} D^{-1}\right). $$
\end{cor}
\begin{proof}
See page 14 of \cite{Koev2012}, Theorem 6.1. A factor of $\beta$ is lost due to differences in nomenclature. The best software to calculate this is described in Koev and Edelman\cite{Koev2006}, {\tt mhg}. Convergence is improved using formula (2.6) in Koev\cite{Koev2012}.
\end{proof}

\begin{cor}
The distribution of $\lambda_{min}$ for the $\beta$-Wishart ensemble with general covariance in diagonal $D$, $P(\lambda_{min}<x)$, is:
$$  1 - \exp\left({\rm trace}(-xD^{-1}/2)\right)\sum_{k=0}^{nt}\sum_{\kappa\vdash k,\kappa_1\leq t}\frac{C_\kappa^{(\beta)}(xD^{-1}/2)}{k!}.  $$
It is only valid when $t = (m-n+1)\beta/2-1$ is a nonnegative integer.
\end{cor}
\begin{proof}
See page 14-15 of \cite{Koev2012}, Theorem 6.1. A factor of $\beta$ is lost due to differences in nomenclature. The best software to calculate this is described in Koev and Edelman\cite{Koev2006}, {\tt mhg}.
\end{proof}

\cite{Koev2012} Theorem 6.2 gives a formula for the distribution of the trace of the $\beta$-Wishart ensemble.

The Figures 1-4 demonstrate the correctness of Corollaries 4 and 5, which are derived from Theorem 3.

\begin{figure}
\centering
\includegraphics[scale=.6]{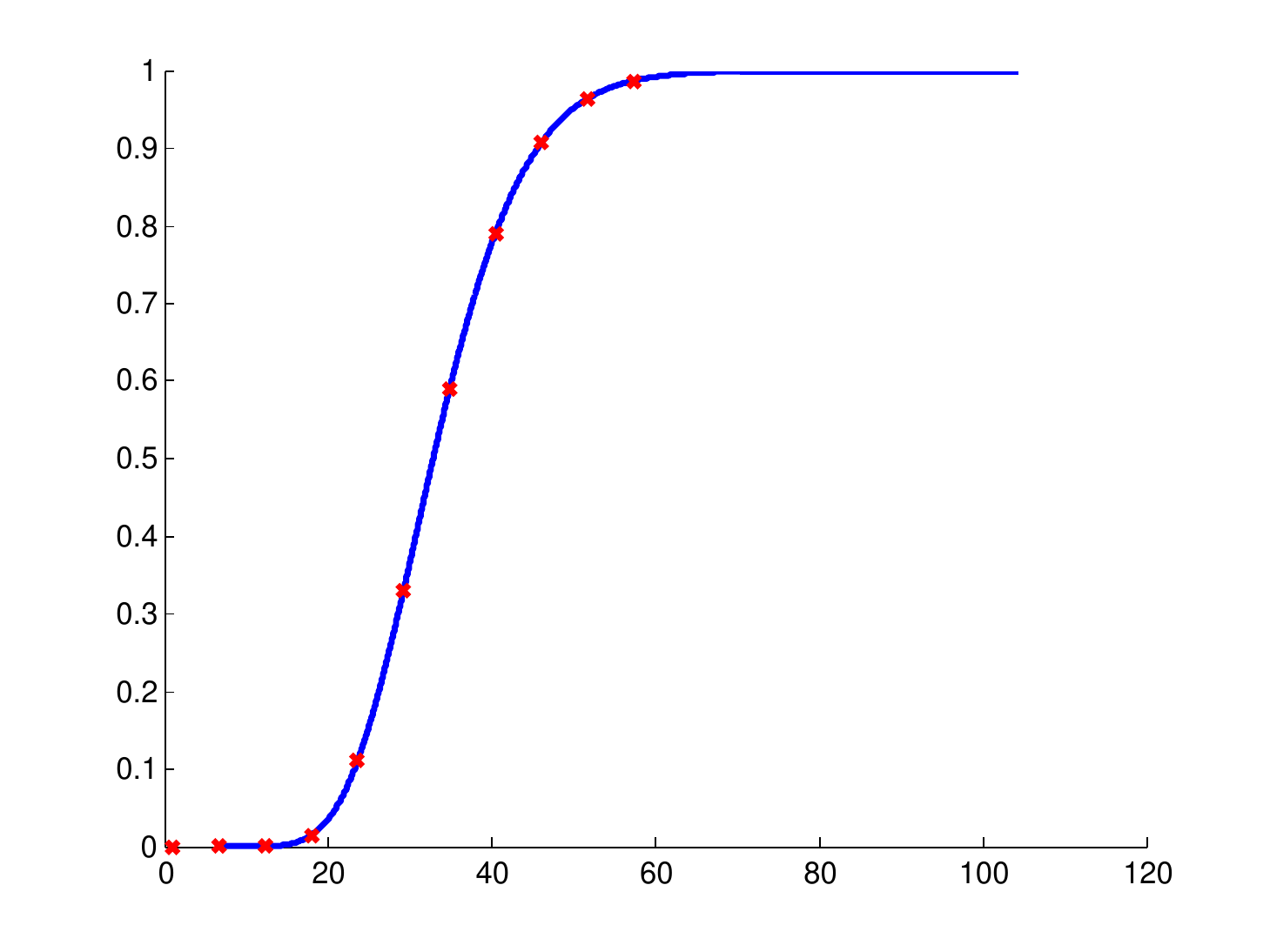}
\caption{The line is the empirical cdf created from many draws of the maximum eigenvalue of the $\beta$-Wishart ensemble, with $m = 4$, $n = 4$, $\beta = 2.5$, and $D = {\rm diag}(1.1,1.2,1.4,1.8)$. The x's are the analytically derived values of the cdf using Corollary 4 and {\tt mhg}.}
\end{figure}

\begin{figure}
\centering
\includegraphics[scale=.6]{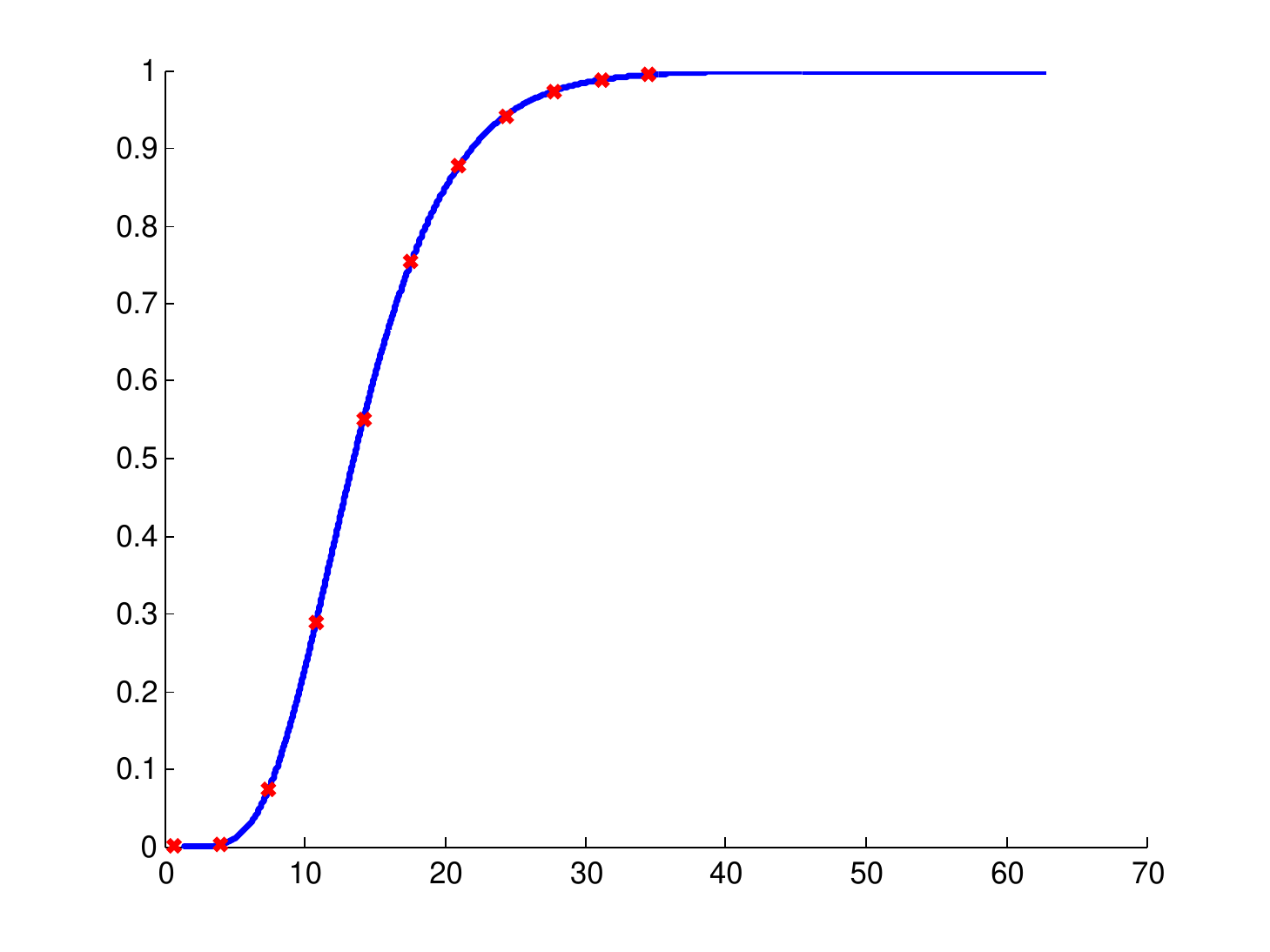}
\caption{The line is the empirical cdf created from many draws of the maximum eigenvalue of the $\beta$-Wishart ensemble, with $m = 6$, $n = 4$, $\beta = 0.75$, and $D = {\rm diag}(1.1,1.2,1.4,1.8)$. The x's are the analytically derived values of the cdf using Corollary 4 and {\tt mhg}.}
\end{figure}

\begin{figure}
\centering
\includegraphics[scale=.6]{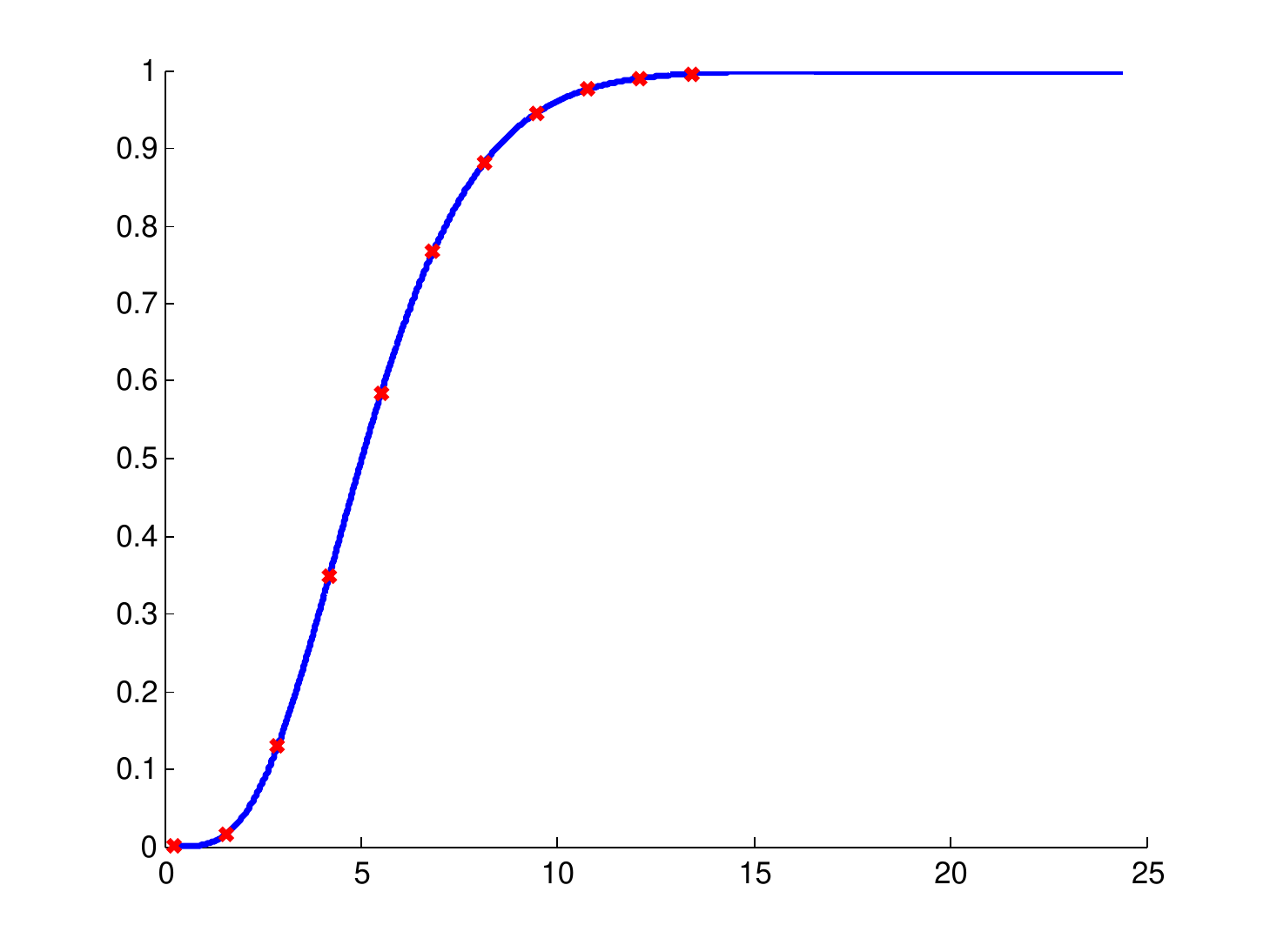}
\caption{The line is the empirical cdf created from many draws of the minimum eigenvalue of the $\beta$-Wishart ensemble, with $m = 4$, $n = 3$, $\beta = 5$, and $D = {\rm diag}(1.1,1.2,1.4)$. The x's are the analytically derived values of the cdf using Corollary 5 and {\tt mhg}.}
\end{figure}

\begin{figure}
\centering
\includegraphics[scale=.6]{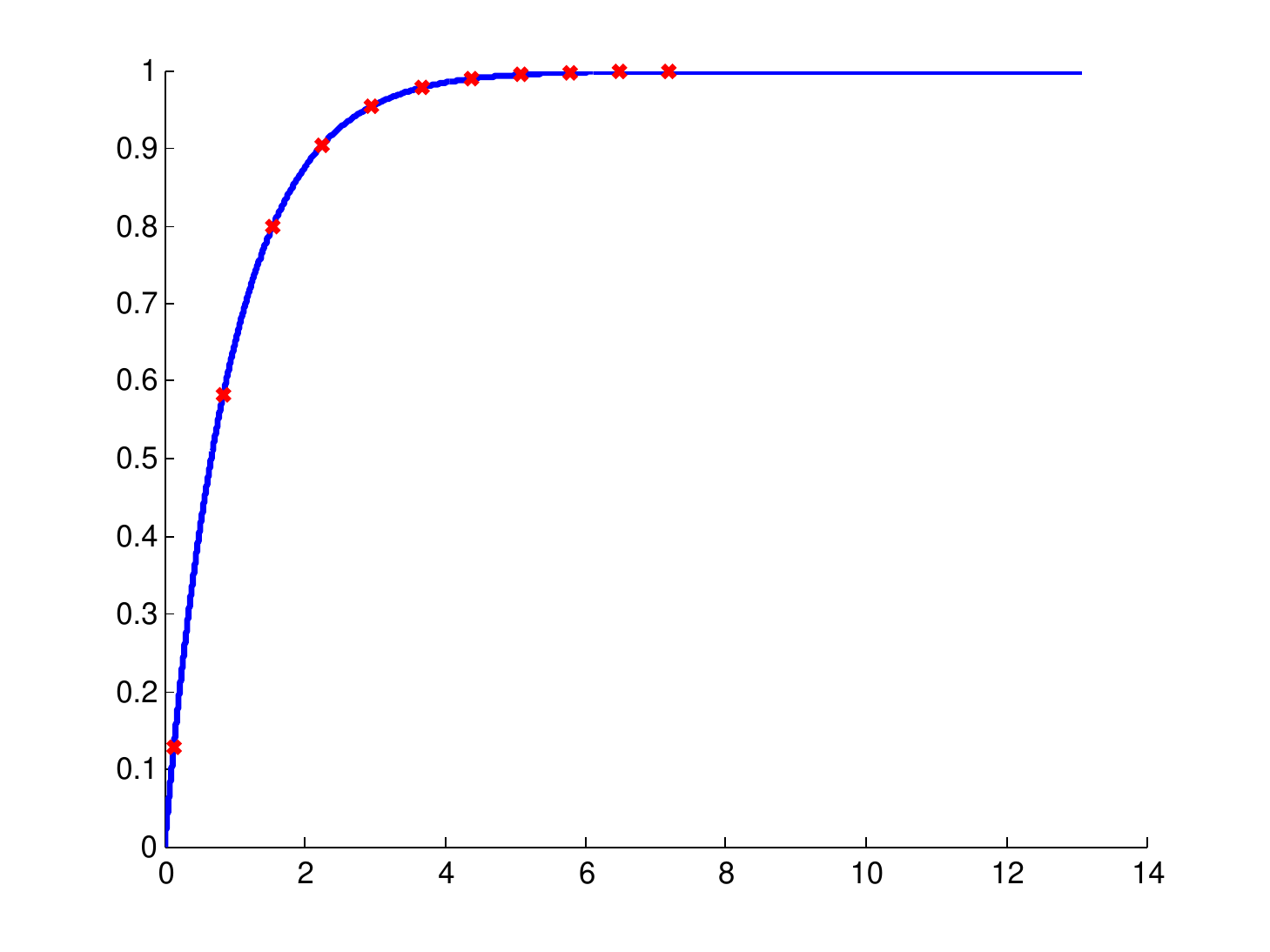}
\caption{The line is the empirical cdf created from many draws of the minimum eigenvalue of the $\beta$-Wishart ensemble, with $m = 7$, $n = 4$, $\beta = 0.5$, and $D = {\rm diag}(1,2,3,4)$. The x's are the analytically derived values of the cdf using Corollary 5 and {\tt mhg}.}
\end{figure}

\section{The $\beta$-Wishart Ensemble and Free Probability}

Given the eigenvalue distributions of two large random matrices, free probability allows one to analytically compute the eigenvalue distributions of the sum and product of those matrices (a good summary is Nadakuditi and Edelman\cite{Rao2007}). In particular, we would like to compute the eigenvalue histogram for $X^tXD/(m\beta)$, where $X$ is a tall matrix of standard normal reals, complexes, quaternions, or Ghosts, and $D$ is a positive definite diagonal matrix drawn from a prior. Dumitriu\cite{Dumitriu2003} proves that for the $D = I$ and $\beta = 1,2,4$ case, the answer is the Marcenko-Pastur law, invariant over $\beta$. So it is reasonable to assume that the value of $\beta$ does not figure into ${\rm hist}({\rm eig}(X^tXD))$, where $D$ is random.

We use the methods of Olver and Nadakuditi\cite{Olver2013} to analytically compute the product of the Marcenko-Pastur distribution for $m/n\longrightarrow 10$ and variance $1$ with the Semicircle distribution of width $2\sqrt{2}$ centered at $3$. Figure 5 demonstrates that the histogram of $1000$ draws of $X^tXD/(m\beta)$ for $m = 1000$, $n = 100$, and $\beta = 3$, represented as a bar graph, is equal to the analytically computed red line. The $\beta$-Wishart distribution allows us to draw the eigenvalues of $X^tXD/(m\beta)$, even if we cannot sample the entries of the matrix for $\beta = 3$.

\begin{figure}
\begin{center}
\includegraphics[scale=1]{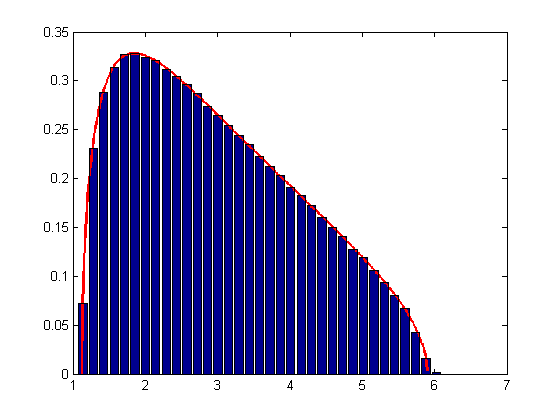}
\end{center}
\caption{The analytical product of the Semicircle and Marcenko-Pastur laws is the red line, the histogram is $1000$ draws of the $\beta$-Wishart ($\beta = 3$) with covariance drawn from the shifted semicircle distribution. They match perfectly.}
\end{figure}

\section{Acknowledgements}

We acknowledge the support of National Science Foundation through grants SOLAR
Grant No. 1035400, DMS-1035400, and DMS-1016086. Alexander Dubbs was funded by the NSF GRFP.

We also acknowledge the partial support by the Woodward Fund for Applied Mathematics at San Jose State University, a gift from the estate of Mrs.\ Marie Woodward in memory of her son, Henry Tynham Woodward. He was an alumnus of the Mathematics Department at San Jose State University and worked with research groups at NASA Ames.


\begin{thebibliography}{99}

\bibitem{Baker1997} T. H. Baker and P. J. Forrester, ``The Calogero-Sutherland model and generalized classical polynomials,'' Communications in Mathematical Physics 188 (1997), no. 1, 175-216.

\bibitem{Bekker1995} A. Bekker and J.J.J. Roux, ``Bayesian multivariate normal analysis with a Wishart prior,'' \textit{Communications in Statistics - Theory and Methods}, 24:10, 2485-2497.

\bibitem{Bunch1978} James R. Bunch and Christopher P. Nielson, ``Updating the Singular Value Decomposition,'' Numerische Mathematik, 31, 111-129, 1978.

\bibitem{Chafai} Djalil Chafai, ``Singular Values of Random Matrices,'' notes available online.

\bibitem{Dumitriu2002} Ioana Dumitriu and Alan Edelman, ``Matrix Models for Beta Ensembles,'' Journal of Mathematical Physics, Volume 43, Number 11, November, 2002.

\bibitem{Dumitriu2003} Ioana Dumitriu, ``Eigenvalue Statistics for Beta-Ensembles,'' Ph.D. Thesis, MIT, 2003.

\bibitem{Dumitriu2007} Ioana Dumitriu, Alan Edelman, Gene Shuman, ``MOPS: Multivariate orthogonal polynomials (symbolically),'' Journal of Symbolic Computation, 42, 2007.

\bibitem{Dyson1962} Freeman J. Dyson, ``The Threefold Way. Algebraic Structure of Symmetry Groups and Ensembles in Quantum Mechanics,'' Journal of Mathematical Physics, Volume 3, Issue 6.

\bibitem{Edelman2005} Alan Edelman, N. Raj Rao, ``Random matrix theory,'' \textit{Acta Numerica}, 2005.

\bibitem{Sutton2007} Alan Edelman and Brian Sutton, ``The Beta-Jacobi Matrix Model, the CS decomposition, and generalized singular value problems,'' Foundations of Computational Mathematics, 2007.

\bibitem{Edelman2009} Alan Edelman, ``The Random Matrix Technique of Ghosts and Shadows,'' Markov Processes and Related Fields, 16, 2010, No. 4, 783-790.

\bibitem{Evans1964} I. G. Evans, ``Bayesian Estimation of Parameters of a Multivariate Normal Distribution'', \textit{Journal of the Royal Statistical Society. Series B (Methodological)}, Vol. 27, No. 2 (1965), pp. 279-283.

\bibitem{Forrester1994} Peter Forrester, ``Exact results and universal asymptotics in the Laguerre random matrix ensemble,'' J. Math. Phys. 35, (1994).

\bibitem{Forrester2005} Peter Forrester, \textit{Log-gases and random matrices}, Princeton University Press, 2010.

\bibitem{Gu1994} Ming Gu and Stanley Eisenstat, ``A stable and fast algorithm for updating the singular value decomposition,'' Research Report YALEU/DCS/RR9-66, Yale University, New Haven, CT, 1994.

\bibitem{Hwang2004} Suk-Geun Hwang, ``Cauchy's Interlace Theorem for Eigenvalues of Hermitian Matrices,'' \textit{The American Mathematical Monthly}, Vol. 111, No. 2 (Feb., 2004), pp. 157-159.

\bibitem{James} A.T. James, ``The distribution of the latent roots of the covariance matrix,'' Ann. Math. Statist., 31, 151-158, 1960.

\bibitem{Johansson} Kurt Johansson, Eric Nordenstam, ``Eigenvalues of GUE minors,'' \textit{Electronic Journal of Probability}, Vol. 11 (2006), pp. 1342-1371.

\bibitem{Killip2004} Rowan Killip and Irina Nenciu, ``Matrix models for circular ensembles,'' International Mathematics Research Notes, Volume 2004, Issue 50, pp. 2665-2701.

\bibitem{Koev2006} Plamen Koev and Alan Edelman, ``The Efficient Evaluation of the Hypergeometric Function of a Matrix Argument,'' Mathematics of Computation, Volume 75, Number 254, January, 2006.

\bibitem{Koev2012} Plamen Koev, ``Computing Multivariate Statistics,'' online notes at http://math.mit.edu/$\sim$plamen/files/mvs.pdf

\bibitem{PlamenURL} Plamen Koev's web page: http://www-math.mit.edu/$\sim$plamen/software/mhgref.html

\bibitem{Li2009} Fei Li and Yifeng Xue, ``Zonal polynomials and hypergeometric functions of quaternion matrix argument,'' Communications in Statistics: Theory and Methods, Volume 38, Number 8, January 2009.

\bibitem{Lippert2003} Ross Lippert, ``A matrix model for the $\beta$-Jacobi ensemble,'' Journal of Mathematical Physics 44(10), 2003.

\bibitem{Muirhead1982} Robb J. Muirhead, \textit{Aspects of Multivariate Statistical Theory}, Wiley-Interscience, 1982.

\bibitem{Okounkov1997} A. Okounkov and G. Olshanksi, ``Shifted Jack Polynomials, Binomial Formla, and Applications,'' Mathematical Research Letters 4, 69-78, (1997).

\bibitem{Olver2013} S. Olver and R. Nadakuditi,``Numerical computation of convolutions in free probability theory,'' preprint on arXiv:1203.1958.

\bibitem{Parlett1998} B. N. Parlett, \textit{The Symmetric Eigenvalue Problem}. SIAM Classics in Applied Mathematics, 1998.

\bibitem{Rao2007} N. Raj Rao, Alan Edelman, ``The Polynomial Method for Random Matrices,'' \textit{Foundations of Computational Mathematics}, 2007.

\bibitem{Ratnarajah2004} T. Ratnarajah, R. Vaillancourt, M. Alvo, ``Complex Random Matrices and Applications,'' CRM-2938, January, 2004.

\bibitem{Santalo1976} Luis Santalo, \textit{Integral Geometry and Geometric Probability}, Addison-Wesley Publishing Company, Inc. 1976.

\bibitem{Stanley1989} Richard P. Stanley, ``Some combinatorial properties of Jack symmetric functions,'' Adv. Math. 77, 1989.

\bibitem{Trefethen1997} Lloyd N. Trefethen and David Bau, III, \textit{Numerical Linear Algebra}, SIAM, 1997.

\bibitem{Wilkinson1999} J. H. Wilkinson, \textit{The Algebraic Eigenvalue Problem}, Oxford University Press, 1999.

\end{thebibliography}

\end{document}